\documentclass[12pt,final]{amsart}

\usepackage{lmodern}
\usepackage{amsfonts,amsmath,amstext,amsbsy,amssymb,
  amsopn,amsthm,amscd,upref,eucal, mathpazo}
\usepackage[T1]{fontenc}
\usepackage{hyperref}
\usepackage{microtype}
\RequirePackage[dvipsnames]{xcolor} 
\definecolor{halfgray}{gray}{0.55} 
\definecolor{webgreen}{rgb}{0,0.5,0}
\definecolor{webbrown}{rgb}{.6,0,0}
\hypersetup{%
    colorlinks=true, linktocpage=true, pdfstartpage=3, pdfstartview=FitV,%
    breaklinks=true, pdfpagemode=UseNone, pageanchor=true, pdfpagemode=UseOutlines,%
    plainpages=false, bookmarksnumbered, bookmarksopen=true, bookmarksopenlevel=1,%
    hypertexnames=true, pdfhighlight=/O,
    urlcolor=webbrown, linkcolor=RoyalBlue, citecolor=webgreen, 
    pdftitle={Linearization of Cohomology-free Vector Fields},%
    pdfauthor={Livio Flaminio and Miguel Paternain},%
    pdfsubject={2000 MAthematical Subject Classification: Primary: 37-XX, 37C15, 37C40},%
    pdfkeywords={Cohomological Equations, Greenfield-Wallach and Katok conjectures},%
    pdfcreator={pdfLaTeX},%
    pdfproducer={LaTeX with hyperref}%
}

\newtheorem{theorem}{Theorem}[section]
\newtheorem{lemma}[theorem]{Lemma}
\newtheorem{corollary}[theorem]{Corollary}
\newtheorem{proposition}[theorem]{Proposition}

\newtheorem{conjecture}{Conjecture}
\newtheorem{definition}[theorem]{Definition}

\theoremstyle{definition}

\newtheorem{remark}[theorem]{Remark}

\newcommand{\field}[1]{\mathbb{#1}}
\newcommand{\R}{\field{R}}

\newcommand{\Z}{\field{Z}}

\makeatletter
\def\arXiv#1{\@ifundefined{href}
 {{\mdseries\ttfamily arxiv:#1}}
  {\href{http://arxiv.org/pdf/#1}
   {{\mdseries available on ArXiv}}}}

\makeatother
\begin{document}
\baselineskip=16pt

\title{Linearization of Cohomology-free Vector Fields}

\author{Livio Flaminio}

\author{Miguel Paternain}

\address{UFR de Math\'ematiques\\
  Universit\'e de Lille 1 (USTL)\\
  F59655 Villeneuve d'Asq CEDEX\\
  FRANCE}

\address{Centro de Matem\'atica\\
  Facultad de Ciencias\\
  Igu\'a 4225\\
  11400 Montevideo\\
  URUGUAY}

\email {livio.flaminio@math.univ-lille1.fr}

\email {miguel@cmat.edu.uy}


\thanks{Miguel Paternain was partially supported by an
  ANII-grant. Both authors were partially supported by the PREMER
franco-uruguayan grant}
\begin{abstract}
We study the cohomological equation for a smooth vector field on a compact manifold.
We show that if the vector field is cohomology free, then it can be embedded continuously in
a linear flow on an Abelian group.
\end{abstract}
\maketitle
\section{Introduction}
\label{sec:intro}

We consider a smooth flow~$(\phi^t)_{t\in \R}$   generated
by a vector field~$X$ on a smooth connected
compact manifold~$M$.

A major problem arising in many different contexts of the theory
of dynamical systems is to solve the \emph{cohomological equation}  for the
flow~$(\phi^t)_{t\in \R}$, the equation given~by
\begin{equation}
  \label{eq:cohom0}
  {L}_X h = f.
\end{equation}
Here $L_X$ denotes the Lie derivative in the direction of the
vector field~$X$, $f$ is a given function and $h$~is~the solution we seek.

To
make sense of this problem it is of course necessary to impose some
regularity conditions on the data~$f$ as well as on the
solution~$h$. In low regularity we shall interpret the
equation~\eqref{eq:cohom0} in a weak sense.

\begin{sloppypar}
  We endow the space $C^\infty(M)$ with the $C^\infty$-topology and
define the \emph{$C^\infty$-cohomology of the
  flow~$(\phi^t)_{t\in \R}$} as the quotient vector space
$C^\infty(M)/ L_X(C^\infty(M)) $; the \emph{reduced
  $C^\infty$-cohomology} is instead the topological vector space
$C^\infty(M)/ \overline {L_X(C^\infty(M))} $, where the closure
is taken in the $C^\infty$-topology.
\end{sloppypar}

Restricting Katok's definition, (\cite{MR1858535}, \cite{MR2008435}),
to the $C^\infty$ setting, we say that the flow~$(\phi^t)_{t\in \R}$
 is \emph{$C^\infty$-stable} if its cohomology and reduced
cohomology coincide, that is if the image of the Lie derivative
operator $L_X$ is a closed subspace of $C^\infty(M)$. By
Hahn-Banach's Theorem this is equivalent to saying that, for every
function~$f$ belonging to the kernel of all $X$-invariant Schwartz
distributions, the equation~\eqref{eq:cohom0} admits a solution~$h
\in C^\infty$.

We remark that since a continuous flow on a compact
manifold admits always an invariant measure, the reduced
cohomology of a flow is at least one-dimensional.

Stability in the $C^\infty$ setting has been established for a
variety of flows; we mention just a few: \cite{MR2003124},
\cite{MR2261072}, \cite{MR2261071}, in the homogeneous setting,
and \cite{MR579578}, \cite{MR573432}, \cite{MR766961} for systems
of dynamical origin. Furthermore stability for action of higher
rank Abelian group has also been established in many exemplary
cases (cf.\ for example \cite{MR2119035}, \cite{MR2166714},
\cite{MR2342703}, \cite{MR2324486}, \cite{MR2261072}).

In all the cases studied so far, the flow cohomology is infinite
dimensional, with one notable exception. In fact, linear flows on
$d$-dimensional tori provide the classical and only known examples
with one-dimensional reduced flow cohomology. It is well known that
 such a flow is $C^\infty$-stable if and only if the direction
numbers $\alpha \in \R^d$ of the vector field~$X$ is a
\emph{Diophantine vector}, that is, if there exist positive
constants $C$ and $\tau$ such that
\begin{equation}
  \label{eq:dioph}
  | n \cdot \alpha | \ge C \| n\|^{-\tau}, \quad \forall \, n \in \Z^d \setminus\{0\}.
\end{equation}
When $\alpha$ is not Diophantine then there exists $f\in C^\infty(\mathbb T^d)$, with $\int f \, \hbox{d}\mu =0$,
for which equation~\eqref{eq:cohom0} admits no solution in
$L^2(\mathbb T^d,\mu)$ (\cite{MR1858535}, \cite{MR2008435}) of
even measurable (\cite{MR2104585}); thus fast approximation
by periodic flows provides one of mechanisms through which
$C^\infty$-stability fails to hold.

A $C^\infty$-stable flow is called \emph{rigid} or
\emph{cohomology free} if its flow cohomology is one dimensional;
A.~Katok (\cite{MR805843}, \cite{MR1858535}, \cite{MR2008435})
suggested the following:

\begin{conjecture}
  \label{problem}
  Let $\{\phi^t\}_{t\in \R}$ be a cohomology-free flow generated
  by a vector field~$X$ on a compact connected
  manifold~$M$. Then, up to diffeomorphism, the manifold $M$~is a
  torus and $X$~is a Diophantine vector field.
\end{conjecture}

Chen and Chi \cite{MR1737551} have essentially proved that Katok's conjecture is equivalent
to a conjecture formulated by Greenfield and Wallach in \cite{MR0320502} and which states that a globally
hypoelliptic vector field is a linear diophantine flow on the torus
(see~\cite{MR2478471}
for a review of the relation betwwen the two conjectures).

Progress towards this conjecture has been limited. A major
advance has been made by F.~and~J.~Rodriguez-Hertz
\cite{MR2191392} who showed that a cohomology-free flow on a
manifold $M$ is up to semi-conjugate to a
Diophantine linear flow on a torus of dimension equal to the
first Betti number of~$M$ (the Albanese variety of $M$); furthermore the 
semi-conjugacy is smooth. More
substantial progress has made for low dimensional manifolds. In
\cite{MR1645342} the analogous problem for diffeomorphisms is
solved for tori of dimension four or less. Recently independent
work of G.~Forni~\cite{MR2478471},
A.~Kocsard~\cite{MR2542719} and S.~Matsumoto~\cite{matsumoto}, have proved the Katok-Greenfield-Wallach
 conjecture when $\dim M \le
3$, using Taubes' proof of Weinstein's Conjecture~\cite{MR2350473}.

More recently Avila and Kocsard have recently announced in \cite{avila_kocsard} that the reduced cohomology of every minimal
circle diffeomorphism --- hence of very minimal flow on the
two-torus --- is one-dimensional; from this it follows easily
that,  up to a diffeomorphism, the only $C^\infty$-stable  minimal flows on the
two-torus are the diophantine linear flows.

\begin{theorem}\label{thm:linear}
  Let $\{\phi^t\}_{t\in \R}$ be a cohomology-free flow generated
  by a vector field~$X$ on a compact connected manifold~$M$.
  Then there are a (possibly non separated) topological Abelian
  group~$A$, a continuous homomorphism $t\in \R \mapsto a t \in
  A$, and a continuous injection $l:M\rightarrow A$, such that
  $l(\phi^t(y))=l(y)+ a t$ for every $y\in M$ and $t\in
  \R$. Furthermore there is a continuous projection $\pi$ of $A$
  onto the Albanese torus of $M$ such that $ \pi \circ l$ is
  F.~and~J.~Rodriguez-Hertz' semi-conjugacy.
\end{theorem}

\subsubsection*{Acknowledgments}
\label{sec:linear:acknowledgement} The authors would like to thank
A. Verjowsky and the Instituto de Matem\'aticas de Cuernavaca for
their hospitality. The second author is grateful to Armando Treibich
for his invitation to the Universit\'e d'Artois and to Federico
Rodriguez Hertz for explaining his work on this problem.

\section{Curves and Loops}
\label{first}

From now our standing hypothesis is that $(\phi^t)_{t\in \R}$ is
a cohomology-free flow generated by a vector field~$X$ on a
compact connected manifold~$M$.

Let
$I$ denote the unit interval $[0,1]$. With the term
``curves'', we shall mean
piecewise $C^1$-immersed parametrized curves $I\to M$, modulo the equivalence
relation given by piecewise
$C^1$-reparametrization.
The initial point
and end point of a curve $\gamma$ are the points
$\alpha(\gamma)=\gamma(0)$ and $\omega(\gamma)=\gamma(1)$.

Two curves $\gamma_1$ and $\gamma_2$ can be concatenated if
$\omega(\gamma_1)=\alpha(\gamma_2)$; with $\gamma=\gamma_1\gamma_2$
we shall denote the usual concatenation $\gamma(t)=\gamma_1(2t)$ for
$t\le 1/2$ and $\gamma(t)=\gamma_2(2t-1)$ for $t\ge 1/2$, $t\in I$.

For the sequel we restrict our consideration to the set $\Gamma$ of
curves which are finite concatenations of unoriented flow
segments and geodesic segments (for some fixed Riemannian
metric on~$M$) transverse to the flow. We endow $\Gamma$ with the
topology of uniform convergence (say for the uniform speed
parametrization of curves).

Curves in  $\Gamma$  have the following regularity property:

\begin{lemma}\label{lem:regular}
  Two curves in $\Gamma$ meet a finite collections of intervals
  (which may reduce to points). Furthermore there is a compact
  set $K\subset \Gamma$ such that every two points of $M$ can be
  joined by an element of $K$.
\end{lemma}
\begin{proof}
  The first statement is an immediate consequence of the choice of
  $\Gamma$: if two geodesic segments intersect, they do so on an
  interval; the same is true for flow segments. For the set $K$
  we can choose the set paths which are concatenation of $N$
  geodesic segments of length less than some positive $\epsilon_0$, where
  $N= [2 \epsilon_0^{-1}\operatorname{diam}(M)]+1$.
\end{proof}

For $x\in M$, we denote by $\Gamma _x$ the curves $\gamma \in
\Gamma$ with $\alpha(\gamma) = x$, that is the curves starting at
$x$. Finally we let $\Delta$ be the set of curves $\gamma\in \Gamma$
with $\alpha(\gamma)=\omega(\gamma)$, the set of closed loops in
$\Gamma$.

\section{Currents}
\label{currents}

Let $\Omega^1(M)$ denote the Fr\'echet space of $C^\infty$
differential one-forms with the $C^\infty$-topology and let
$\mathcal C^{0}(M) $ and $\mathcal C^{1}(M)$ be the
Fr\'echet spaces of de Rham currents on~$M$ of degree zero
and one, that is the dual space of $\Omega^0(M)=C^\infty(M)$
and the dual space of $\Omega^1(M)$, respectively. The
spaces $\mathcal C^{0}(M) $ and $\mathcal C^{1}(M)$ are
endowed with the vague topology.

To each parametrized curve  $\gamma$ it corresponds an integration
one-current $\widetilde \gamma$ given~by
\[
\widetilde \gamma(\eta) = \int_\gamma \eta, \qquad \forall \eta\in \Omega^1(M).
\]
Clearly the current $\widetilde \gamma$ does not depend upon the
choice of a parametrization of $\gamma$ so that the map $\gamma
\in \Gamma \to \widetilde \gamma \in \mathcal C^{1}(M)$ is well
defined and we set
\[
\widetilde \Gamma = \{\widetilde \gamma \in  \mathcal C^{1}(M)
\mid  \gamma
\in \Gamma\};
\]
we also denote by~${\widetilde \Gamma}_x$
  and~${\widetilde \Delta}$ the sets of currents images of
  curves in~$\Gamma_x$ and loops~in~${\Delta}$.

It is obvious from the definition that for any  $\gamma_1$ and
$\gamma_2$ in $\widetilde \Gamma$ for which the concatenation is defined we have
\[ \widetilde{(\gamma_1\gamma_2)} =
\widetilde \gamma_1+\widetilde \gamma_2.\]

The following  elementary propositions are  stated here for further
reference. We set $\mathcal M= \{ \delta_x\in \mathcal C^0(M) \mid x\in
M\}$.

\begin{proposition}
  \label{def:firsthomeo}
  The map $x\in M\mapsto \delta_x\in \mathcal
  M$ is continuous (in fact differentiable) and injective
  and hence a homeomorphism  $M \approx \mathcal M$.
\end{proposition}

\begin{proposition} For any $x\in M$ we have
  \label{prop:kernel}
  \begin{equation}
  \label{eq:lambda2} \widetilde\Delta= \{ T\in {\widetilde \Gamma} \mid
\partial T=0\} = \widetilde{\Gamma_x\cap \Delta}.
\end{equation} Hence ${\widetilde \Delta}$ is a closed set both in ${\widetilde \Gamma}$ and
in ${\widetilde \Gamma}_x$. Furthermore  $
{\widetilde \Gamma}$ and ${\widetilde \Gamma}_x$ are both invariant by translation
by ${\widetilde \Delta}$.
\end{proposition}
\begin{proof} Let $T={\widetilde \gamma}$ with $\gamma\in
  \Gamma$.
Since $\partial T = \delta_{\omega(\gamma)}- \delta_{\alpha(\gamma)}$
we have $\partial T = 0$ iff $\alpha(\gamma)=\omega(\gamma)$, i.e.\
iff $\gamma\in \Delta$. This shows that $\widetilde\Delta= \{
T\in {\widetilde \Gamma} \mid
\partial T=0\}$.

Suppose that $T={\widetilde \gamma}$ with $\gamma\in
  \Delta$. Letting
$\delta\in \Gamma$ be a curve such that $\alpha(\delta)= x$ and
$\omega(\delta) = \alpha(\gamma)$, we have $\gamma_1= \delta \gamma
\delta^{-1} \in \Delta \cap  \Gamma_x$ and $\widetilde
{\gamma_1}={\widetilde \gamma}=T$. Thus $T\in  \widetilde{\Gamma_x\cap \Delta}$ and
$\widetilde\Delta=  \widetilde{\Gamma_x\cap \Delta}$.

If $\gamma \in \Gamma_x$ and $\delta \in \Delta$ let $\gamma_1\in \Gamma_x$ be
a curve with $\omega(\gamma_1) = \alpha(\delta)$; then the   curve
$\gamma_1 \delta \gamma_1^{-1} \gamma$ belongs to  $\Gamma_x$ and
the current associated to it is equal to  $\widetilde
{\gamma} + {\widetilde \delta}$. This shows that  ${\widetilde \Gamma}_x$ (and
hence ${\widetilde \Gamma}$)  is invariant by translation by
$\widetilde{\delta} \in \widetilde{\Delta}$.
\end{proof}

\begin{remark}
  Observe that the above proposition implies that ${\widetilde
    \Delta}$ is an additive subgroup of space of the currents
  $\mathcal C^1(M)$. The Abelian group $A$ appearing in
  Theorem~\ref{thm:linear} will be defined as the quotient group
  $\mathcal C^1(M)/{\widetilde \Delta}$.
\end{remark}
\begin{proposition}
  \label{prop:repcurve} For $x\in M$, let $\pi_x: \mathcal C^1(M)\to
   \mathcal C^0(M)$ be the continuous affine map defined by
  \[\pi_x(T)=\partial T + \delta_x.\]
  For $\gamma\in \Gamma_x$, we have
  $\pi_x(\widetilde\gamma)= \delta_{\omega(\gamma)}$ and
  hence $\pi_x$ maps ${\widetilde \Gamma}_x$ onto $\mathcal M$.

  Furthermore, if $T_1, T_2\in {\widetilde \Gamma}_x$, we have
  $\pi_x(T_1) = \pi_x(T_2)$ if and only if $T_1-T_2 \in
  {\widetilde \Delta}$.
\end{proposition}

\begin{proof} The first affirmation is obvious.
Suppose that $\pi_x(T_1) = \pi_x(T_2)$ where $T_i={\widetilde{\gamma}_i}$, $\gamma_i\in
  \Gamma_x$. Then
  $\gamma_1^{-1}\gamma_2\in \Gamma$, thus $(T_1-T_2) \in
  {\widetilde \Gamma}$ and $\partial (T_1-T_2) =0$. By Proposition~\ref{prop:kernel} we have $(T_1-T_2)\in {\widetilde \Delta}$.
\end{proof}

\begin{corollary}
  \label{coro:homeo3}
  The map $\pi_x$ induces a homeomorphism   \[p_x:
  {{\widetilde \Gamma}_x}/{\widetilde \Delta} \to \mathcal M.\] Hence
  \[
   {{\widetilde \Gamma}_x}/{\widetilde \Delta}
  \approx \mathcal M \approx M.
  \]
\end{corollary}
\begin{proof}
  The induced map $p_x$ is continuous for the quotient
  topology  and injective  by Proposition~\ref{prop:repcurve}. Let $K$ be as in Lemma~\ref{lem:regular}.  The set of currents
  images of elements in $K\cap \Gamma_x$ is compact in the weak topology
  and surjects onto ${{\widetilde \Gamma}_x}/{\widetilde
    \Delta}$; hence ${{\widetilde \Gamma}_x}/{\widetilde
    \Delta}$ is compact. Since $\mathcal M$ is a Hausdorff
  space, the map $p_x$ is a homeomorphism.
\end{proof}

\section{Twisting the embedding}

The hypothesis of Theorem~\ref{thm:linear} imply that for every
$\eta\in \Omega^1(M)$ there exist a $C^\infty$~function $h_\eta : M
\to \R$ and a constant $c_\eta\in \R$ such that
\begin{equation}
  \label{eq:cohomup} L_X h_\eta =\eta( X) - c_\eta.
\end{equation} The function $h_\eta$ is only defined up to
a constant, by the unique ergodicity of the flow
$(\phi^t)$. We shall need the following observations.
\begin{remark}
  \label{rem:changeform} If $\eta_1, \eta_2\in
\Omega^1(M)$ and $ \eta_1( X)-\eta_2( X)= C$ is a
constant function then $h_{\eta_1}- h_{\eta_2}$ is a
constant and $c_{\eta_1}-c_{\eta_2}=C$.
\end{remark}

\begin{proposition}
  \label{rem:opmapth} Let $C^\infty_0 (M)$ be the quotient
space $C^\infty(M)$ modulo constants (which we can also
identify to the space $C^\infty$ functions of $\mu$-average
zero). Then the maps
  \[
  \eta\in \Omega^1(M) \mapsto h_\eta \in C_0^\infty (M)
  \]
  and
  \[
  \eta\in \Omega^1(M) \mapsto dh_\eta \in \Omega^1(M)
  \]
  are continuous if $C_0^\infty (M)$ and $\Omega^1(M) $ are
endowed with the $C^\infty$ Fr\'echet topology.
\end{proposition}
\begin{proof} The Lie derivative $L_X: C^\infty_0 (M) \to
C^\infty_0 (M)$, is a continuous linear operator which, by
hypothesis and by the unique ergodicity of $\mu$, is also
bijective.  Since $ C^\infty_0 (M) $ is a Fr\'echet space,
the open mapping theorem implies that this map is an
isomorphism.  The map $\Omega^1(M) \to C^\infty (M) $,
$\eta \mapsto \eta(X)$, being clearly continuous, our
claim follows.
\end{proof}

We shall use the family of functions $h_\eta$ to ``twist''
our homeomorphism $M\approx {\widetilde
  \Gamma_x}/{\widetilde \Delta}$.

\begin{definition}
  \label{def:lhat}
  We define the map  $L:  \mathcal C^1(M)\to  \mathcal C^1(M)$
  \begin{equation}
    \label{eq:hatl}
    L(T)(\eta)=
    T(\eta)- \partial T( h_{\eta}) = T(\eta -
    dh_{\eta} )
  \end{equation}
 (which is well defined since $h_\eta$ is defined up to a
 constant).
\end{definition}

The following lemma is obvious.

\begin{lemma}
  \label{def:lhat1}
  The map $L$ is a continuous linear operator (in fact a
  projection). The restriction of $L$ to ${\widetilde \Delta}$ is the
  identity map of ${\widetilde \Delta}$.
\end{lemma}

Let $L_x$ be the restriction of $L$ to ${\widetilde
  \Gamma}_x$.

\section{Injectivity of~$L_x$}

The aim of this section is to show the following proposition.
\begin{proposition}
  \label{prop:inj}
  The map $L_x: {\widetilde \Gamma}_x \to  \mathcal C^1(M)$
  is a continuous injection.
\end{proposition}

\begin{definition}
  \label{def:broadequiv} We say that a finite family $\{
\gamma_1,\ldots, \gamma_m\}\subset \Gamma$ is broadly
equivalent to a finite family $\{ \beta_1,\ldots,
\beta_n\}\subset \Gamma$, if
  \[
  \sum_{i=1}^m {\widetilde \gamma}_i = \sum_{i=1}^n {\widetilde\beta}_i.
  \]
\end{definition}

Clearly broad equivalence is an equivalence relation.

\begin{definition} We say that a curve $\gamma\in \Gamma$
contains a retraced arc~$r$ if $\gamma=arbr^{-1}c$, with
$a$,$b$, $c$ and $r\in \Gamma$.  Similarly, we say that two
given curves $\gamma_1, \gamma_2, \in \Gamma $ contain a
retraced arc~$r$ if $\gamma_1=arb$ and $\gamma_2=cr^{-1}d $
with  $a$,$b$, $c$, $d$ and $r\in \Gamma$.
\end{definition}

From a set of curves in $\Gamma$ we can, in an iterative way, obtain
a new set of curves, broadly equivalent to the given set and without
retraced arcs.

In fact if $\{\gamma=arbr^{-1}c\}$ we say that the set of
curves $\{\gamma_1=ac, \gamma_2=b\} $ is the set {\em
obtained from $\{\gamma\}$ by simple excision of the retraced
arc $r$}. (Observe that $\gamma_2$ is a closed curve).

If a set of two curves is given, $\{\gamma_1, \gamma_2\} $, such
that $\gamma_1=arb$ and $\gamma_2=cr^{-1}d $ we have two cases:
if $\gamma_2$ is closed, the \emph{simple excision of the retraced
  arc $r$} will yield a single curve set $\{\gamma_3=adcb\}$; if
$\gamma_2$ is open the simple excision of $r$ will result in the
set of two curves $\{\gamma_3=ad, \gamma_4=cb\}$.

It is clear in the above procedure that after a simple
excision the new set of curves is broadly equivalent to the
original set.

The simple excision of a retraced arc from one or two arcs
in a family of broken arcs $\gamma^0_1$, $\gamma^0_2,\ldots,
\gamma^0_n$ yields a new family of broken arcs; successive
simple excisions will lead to a family of broken arcs
$\gamma_1$, $\gamma_2,\ldots, \gamma_{n'}$ which we say
\emph{obtained by maximal excision from $\gamma^0_1,
  \gamma^0_2,\ldots, \gamma^0_n$} if it does not contain
any further retraced arcs.

The proof of the following two elementary lemmata is obtained by
induction on the number of excisions and by a direct application of
the definition of excision.

\begin{lemma} If the sequence $G_j=\{ \gamma_1^j, \gamma_2^j,\ldots, \gamma_{n_j}^j\}$ of finite families of curves
  is obtained
by successive excisions from $G_0=\{ \gamma_1^0, \gamma_2^0,\ldots,
\gamma_{n_0}^0\}$ then
  \begin{enumerate}
  \item we have $\partial G_j = \partial G_0$, for all
$j\ge0$, where $\partial G_j =\sum_{i=1}^{n_j} \partial
\gamma_i^j$.
  \item $G_j=\{ \gamma_1^j, \gamma_2^j,\ldots,
 \gamma_{n_j}^j\}$ is broadly equivalent to $G_0$,
for all $j\ge0$.
  \end{enumerate}
\end{lemma}

\begin{lemma} Let $\gamma\in \Gamma$. Then there exists a set of curves obtained by
maximal excision $G=\{\gamma_1$, $\gamma_2,\ldots,
\gamma_{n}\}$ from $\gamma$ such that $\gamma_2, \ldots,
\gamma_n$ are closed and $\gamma_1$ is closed if and only
if $\gamma$ is.
\end{lemma}

The choice of $\Gamma$ as a set of curves obtained as concatenations
of geodesic segments and flow segments yields a simple proof the
following lemma which holds true in greater generality~(cf.\ \cite{MR1473623}).

\begin{lemma}
  \label{lem:regul} Let $G=\{\gamma_1, \gamma_2,\ldots,
\gamma_{n}\}$ be a finite subset of $\Gamma$. Set  $Y=\bigcup_{i=1}^n
\gamma_i(I)$. Recall that a point $y \in Y$ is \emph{regular} if
it  satisfies the following two conditions:
\begin{enumerate}
\item every $t\in \bigcup_{i=1}^n \gamma^{-1}_i\{y\}$ is a
regular point and
\item there exists an open
neighborhood~$W$ of~$y$ such that $Y \cap W$ is a embedded
arc.
\end{enumerate}
The set of regular points
$y \in Y$ is an open and dense subset of $Y$.
\end{lemma}

\begin{proof}[Proof of Proposition~\ref{prop:inj}]

Suppose $L_x({\widetilde \gamma}_1)=L_x({\widetilde \gamma}_2)$ with
$\gamma_1,\gamma_2\in \Gamma_x$.

First, suppose that the points $\omega(\gamma_1)$ and
$\omega(\gamma_2)$ are distinct, i.e. $\partial {\widetilde\gamma}_1\neq
\partial {\widetilde\gamma}_2$.

Let $\gamma$ be the curve $\gamma=\gamma_1^{-1}\gamma_2$;
$\gamma$ is an open curve in $\Gamma$ with $\partial
{\widetilde \gamma}=
  \partial {\widetilde\gamma_2} - \partial {\widetilde\gamma_1}=
\delta_{\omega(\gamma_2)}- \delta_{\omega(\gamma_1)}\neq0$. The
hypothesis  $L_x({\widetilde \gamma_1})=L_x({\widetilde
  \gamma_2})$ yields   $L({\widetilde \gamma})=0$, that is
  \begin{equation}
    \label{eq:cohom2} {\widetilde\gamma}(\eta)
=h_{\eta}(\omega(\gamma))- h_{\eta}(\alpha(\gamma)), \qquad
\forall \eta\in \Omega^1(M).
  \end{equation}

  By considering a maximal excision of the set $\{\gamma\}$, we
  obtain a set of curves
$G=\{\gamma_1 ,\gamma_2, \ldots, \gamma_n\}\subset \Gamma$ satisfying
\begin{enumerate}
\item the set $G$ is
broadly equivalent to $\{\gamma\}$;
\item  $\partial\gamma_1=
\partial\gamma$, $\partial\gamma_2=\cdots=\partial\gamma_n=0$;
\item every $\gamma_i\in G$ admits no further excisions, that is
  no sub-arc of the collection $G$ is retraced.
\end{enumerate}

The first condition
means
\begin{equation} \widetilde \gamma =\sum_i \widetilde \gamma_i.
  \end{equation}

  We have two cases.

  In the first case there exists $t_0\in I$ such that
  $\gamma_i(t_0)$ is regular and the velocity $\dot \gamma(t_0)$
  is not collinear to $X(\gamma_i(t_0))$. By the previous Lemma,
  there exists an open neighborhood $W$ of $\gamma_i(t_0)$, such
  that $W\cap \cup_j \gamma_j(I)= W\cap \gamma_i(I)$ is an
  embedded arc. Since $\gamma_i(I)$ and $X$ are transverse, there
  exists a one-form $\theta$, supported in $W$, such that
  $\theta(X)$ vanishes identically and such that
  $\int_{\gamma_i}\theta\neq 0$.  Since there are no retraced
  arcs in the set $G$ and by the choice of $W$ and $\theta$ we
  have $\sum_{i}\int_{\gamma_i}\theta \neq 0$. By the
  Remark~\ref{rem:changeform} we have that $h_\theta$ is
  identically constant. But this contradicts \eqref{eq:cohom2}
  since $\int_{\gamma}\theta=\sum_i \int_{\gamma_i}\theta\neq 0$
  and $h_{\theta}(\omega(\gamma))-
  h_{\theta}(\alpha(\gamma))=0$. This case is impossible.

We are left with the case where, at all regular points
$\gamma(t_0)$, the velocity $\dot \gamma(t_0)$ is collinear to
$X(\gamma_i(t_0))$.  Then, since the flow of $X$ does not admit
closed orbits the collection of curves $G$ is reduced to
a singleton $\{\gamma_1\}$. Since $\gamma_1$ does not contain retraced arcs it
is in fact a segment of orbit, and we obtain that the original curve
$\gamma$ is broadly equivalent to a segment of orbit $\gamma_1: t\in
[0,1]\mapsto \phi^{st}(x)$, from $\alpha(\gamma_1)=\alpha(\gamma)$ to
$\omega(\gamma_1)=\omega(\gamma)$. Let $\eta_0$ be a one form such that $\eta_0(X)=1$
so that $\int_{\gamma}\eta_0 =\int_{\gamma_1}\eta_0 =s$.  Again, by
the Remark~\ref{rem:changeform}, we observe that $h_{\eta_0}$ is a
constant function. Hence $h_{\eta}(\omega(\gamma))-
h_{\eta}(\alpha(\gamma))=0$, a contradiction if $s\neq 0 $, i.e.\ if
$\alpha(\gamma)\neq\omega(\gamma) $. Thus the second case is also impossible.

We showed that we must have  $\alpha(\gamma)=\omega(\gamma) $. Then the curve
$\gamma=\gamma_1\gamma_2^{-1}$ is a closed loop (based at
$x$); by hypothesis  $L_x({\widetilde
  \gamma_1})=L({\widetilde\gamma_2})$, hence
$L({\widetilde\gamma})=0$. Finally by Lemma~\ref{def:lhat1} we obtain
$0=L({\widetilde\gamma})={\widetilde \gamma}$. Thus
${\widetilde\gamma_1}={\widetilde\gamma_2}$ and we conclude that $L_x$ on
${{\widetilde \Gamma}}_{x}$ is injective.
\end{proof}

\section{Proof of Theorem~\ref{thm:linear}}

Let $x\in M$ be a point fixed once for all.  Consider the Abelian
group $A=\mathcal C^1(M)/\widetilde \Delta$.  The map $L:\mathcal
C^1(M)\to \mathcal C^1(M)$ of Definition~\ref{def:lhat} defines a
quotient map of $A$ into itself since, by Lemma~\ref{def:lhat1},
$L|\widetilde \Delta$ is the identity mapping of $\widetilde
\Delta$. The restriction of $L$ to $\widetilde{\Gamma}_x$, which
we have denoted by~$L_x$, induces a mapping of
$\widetilde{\Gamma}_x /\widetilde \Delta$ into $A$, which, by
Proposition~\ref{prop:inj}, is in fact a continuous bijection of
$ \widetilde{\Gamma}_x /\widetilde \Delta$ onto $L_x(
\widetilde{\Gamma}_x )/\widetilde \Delta$. Using the
identification $M \approx \widetilde{\Gamma}_x /\widetilde
\Delta$ of Corollary~\ref{coro:homeo3} we then conclude
that~$L_x$ induces a continuous injection $l:M\rightarrow A$,
mapping $M$ onto $L_x( \widetilde{\Gamma}_x )/\widetilde \Delta$.

Fix $y\in M$ and let $\gamma$ be a curve starting at $x$ and
ending at $y$. For $t\in \R$, let $\gamma_t$ be the arc of orbit
$ s \mapsto \phi^{ts} (y)$. Integrating the equation
\eqref{eq:cohomup} along the orbit $\gamma_t$ of $y \in M$ we
have
\begin{equation}
  \label{eq:integr1} \int_{\gamma_t} \eta= c_\eta t +
  h_\eta(\phi^t (y)) - h_\eta(y)
\end{equation}
Let $c\in \mathcal C^1(M)$ be given by $c(\eta)=c_{\eta}$. Then
the equations \eqref{eq:integr1}, in view of the
Definition~\ref{def:lhat}, can be rewritten as the following
equation in~$\mathcal C^1(M)$ for the currents $\widetilde\gamma$
and $\widetilde\gamma_t$ associated to the arcs $\gamma$ and
$\gamma_t$:
\[ L(\widetilde\gamma +
\widetilde\gamma_t)=L(\widetilde\gamma)+ct.
\]
As the endpoints of the arcs $\widetilde\gamma \widetilde\gamma_t
$ and $\widetilde\gamma $ are respectively equal to $\phi_t(y)$
and~$y$, passing to the quotient by $\widetilde \Delta$, we
obtain that $l(\phi^t(y))=l(y)+ ta$, where $t\mapsto ta$ is
the projection to~$A$ of the line subgroup $t\mapsto ct $.

Finally notice that, by restricting the space of currents
$\mathcal C^1(M)$ to the finite dimensional space $H^1(M)$ of
harmonic one-forms on~$M$, we obtain a projection $q:\mathcal
C^1(M)\to H^1(M)^*$.  Taking a further quotient by the
lattice~$P$ of periods of~$M$, we obtain a projection~$q'$ of
$\mathcal C^1(M)$ onto the Albanese torus $H^1(M)^*/P$; the
map~$q'$ factors through a map~$\pi: A \to H^1(M)^*/P $, as~$q$
sends the loop group $\widetilde \Delta$ onto~$P$. Thus we have
\[
\mathcal C^1(M) \to \mathcal C^1(M) /\widetilde \Delta \to
H^1(M)^* /P.
\]
Since the maps $L$ and $q'$ are smooth, the composite map $
q'\circ L$ is also smooth. For any smooth local lift $\phi: U
\subset M \to \widetilde \Gamma_x $ of the projection map $\tilde
\Gamma_x \to M$, we have that $q'\circ L \circ \phi= \pi\circ
l\vert_ U$; we conclude that $\pi'=\pi\circ l$ is a smooth map of
$M$ into $H^1(M)^*/P$.

From the continuity of $\pi'=\pi\circ l$ and the minimality of
the flow it follows that the image of $\pi'$ is a rational
sub-torus of $H^1(M)^*/P$. However the map associating to a
closed loop in $M$ the period along this loop induces a
surjection of $\widetilde \Delta$ onto $P$; it follows that
$\pi'$ is surjective in homology and thus (smooth and) surjective
and indeed equal to F.~and~J.  Rodriguez Hertz' semi-conjugacy.

\begin{remark}
  The proof above shows that in fact there is, to some degree, a
  differential structure on the space $ L_x( \widetilde{\Gamma}_x
  )/\widetilde \Delta$, inherited from the linear structure of
  $\mathcal C^1(M)$; in fact one could show that the map $l$ is a
  morphism of differential (or diffeological) spaces in the sense
  of Chen and Iglesias (\cite{MR0454968}, \cite{Iglesias}). As
  the major problem is to tie the topological structure with the
  differentiable one, we omit  any discussion of this
  point.
\end{remark}

\bibliography{biblio} \bibliographystyle{amsalpha}
\end{document}